\documentclass[preprint,12pt]{elsarticle}

\usepackage{amsmath}
\usepackage{amsthm}
\usepackage{mathrsfs}
\usepackage{amsfonts,amssymb}
\usepackage{setspace}
\usepackage{indentfirst}
\usepackage{enumerate, hyperref}
\usepackage{dcolumn}
\usepackage{color}
\hypersetup{colorlinks,urlcolor=blue, citecolor=red}

\numberwithin{equation}{section}

\theoremstyle{plain}
\newtheorem{thm}{Theorem}[section]
\newtheorem{lem}[thm]{Lemma}

\newtheorem{cor}[thm]{Corollary}

\theoremstyle{definition}
\newtheorem{rem}{Remark}

\newcommand\R{{\mathbb R}}

\begin{document}

\begin{frontmatter}



\title{Muliti-scale regularity  of axisymmetric Navier-Stokes equations}
\author{Daoyuan Fang\fnref{a}}
\ead{dyf@zju.edu.cn}
\author{Hui Chen\corref{*}\fnref{b}}
\ead{chenhui@zust.edu.cn}
\author{Ting Zhang\fnref{a}}
\ead{zhangting79@zju.edu.cn}
\address[a]{School of Mathematical Sciences, Zhejiang University,  Hangzhou 310027, P.R. China}
\address[b]{School of Science, Zhejiang University of Science and Technology, Hangzhou 310023, P.R. China }
\cortext[*]{Corresponding author.}
\begin{abstract}
By applying the delicate \textit{a priori} estimates for the equations of  $(\Phi,\Gamma)$, which is introduced in the previous work, we obtain some multi-scale regularity criteria of the swirl component $u^{\theta}$ for the 3D axisymmetric Navier-Stokes equations. In particularly, the solution $\mathbf{u}$ can be continued beyond the time $T$, provided that $u^{\theta}$ satiesfies
$$
u^{\theta} \in L^{p}_{T}L^{q_{v}}_{v}L^{q_{h},w}_{h},~~\frac{2}{p}+\frac{1}{q_{v}}+\frac{2}{q_{h}}\leq 1, ~2<q_{h}\leq\infty,~\frac{1}{q_{v}}+\frac{2}{q_{h}}<1.
$$
\end{abstract}


\end{frontmatter}


\section{Introduction}
This article aims at presenting some new regularity criteria of the swirl component $u^{\theta}$, in the framework of anisotropic Lebesgue space, which improve in that of \cite{Chen2017b}.

Consider the Cauchy problem of the 3D Navier-Stokes equations:
\begin{equation}
\left\{
\begin{split}\label{NS}
&\partial_{t}\mathbf{u}+(\mathbf{u}\cdot\nabla)\mathbf{u}-\Delta \mathbf{u}+\nabla p=\mathbf{0},\\
&\nabla\cdot \mathbf{u}=0~,\\
&\mathbf{u}|_{t=0}=\mathbf{u_{0}}~.
\end{split}
\right.\hspace{10pt}(t,x)\in\R^+\times\R^{3}~,
\end{equation}
where $\mathbf{u}(t,x)=(u^{1},u^{2},u^{3})$, $p(t,x)$ and $\mathbf{u_{0}}$ denote the fluid velocity field,  the pressure, and  the given initial velocity field, respectively.

For given $\mathbf{u_{0}}\in L^{2}(\R^{3})$ with $\mathrm{div}~\mathbf{u_{0}}=0$ in the sense of distribution, a global weak solution $\mathbf{u}$ to the Navier-Stokes equations was
constructed by Leray \cite{Leray} and Hopf \cite{Hopf}, which is called Leray-Hopf weak solution. The regularity of such Leray-Hopf weak solution in three dimension plays an important role in the mathematical fluid mechanics. One essential work is usually referred as Prodi-Serrin (P-S) conditions (see \cite{Escauriaza,Fabes,Giga,Prodi,Serrin,Struwe,Takahashi}), i.e. if in addition, the weak solution $\mathbf{u}$  belongs to $L^{p}((0,T);L^{q}(\R^{3}))$, where $\frac{2}{p}+\frac{3}{q}\leq 1$, $3\leq q\leq\infty$,
then the weak solution becomes regular.

In this paper, we assume that the solution $\mathbf{u}$ of the system (\ref{NS}) has the axisymmetric form
\begin{equation}\label{u}
\mathbf{u}(t,x)=u^{r}(t,r,x_{3})\mathbf{e_{r}}+u^{\theta}(t,r,x_{3})\boldsymbol{e_{\theta}}+u^{3}(t,r,x_{3})\mathbf{e_{3}},
\end{equation}
where
\begin{equation*}
\mathbf{e}_{r}=(\frac{x_{1}}{r},\frac{x_{2}}{r},0),~\mathbf{e}_{\theta}=(-\frac{x_{2}}{r},\frac{x_{1}}{r},0),~\mathbf{e}_{3}=(0,0,1),
\ r=\sqrt{x_{1}^{2}+x_{2}^{2}}.
\end{equation*}
In above, $u^{\theta}$ is usually called the swirl component. And if $u^{\theta}=0$, the solution $\mathbf{u}$ is without swirl.

For the axisymmetric solutions of Navier-Stokes system, we can equivalently reformulate (\ref{NS}) as
\begin{equation}\label{NS-u}
\left\{
\begin{split}
&\partial_{t}u^{r}+(u^{r}\partial_{r}+u^{3}\partial_{3})u^{r}-(\partial_{r}^{2}+\partial_{3}^{2}+\frac{1}{r}\partial_{r}-\frac{1}{r^{2}})u^{r}-\frac{(u^{\theta})^{2}}{r}+\partial_{r}p=0,\\
&\partial_{t}u^{\theta}+(u^{r}\partial_{r}+u^{3}\partial_{3})u^{\theta}-(\partial_{r}^{2}+\partial_{3}^{2}+\frac{1}{r}\partial_{r}-\frac{1}{r^{2}})u^{\theta}+\frac{u^{\theta}u^{r}}{r}=0,\\
&\partial_{t}u^{3}+(u^{r}\partial_{r}+u^{3}\partial_{3})u^{3}-(\partial_{r}^{2}+\partial_{3}^{2}+\frac{1}{r}\partial_{r})u^{3}+\partial_{3}p=0,\\
&\partial_{r}u^{r}+\frac{1}{r}u^{r}+\partial_{3}u^{3}=0,\\
&(u^{r},u^{\theta},u^{3})|_{t=0}=(u_{0}^{r},u_{0}^{\theta},u_{0}^{3}).
\end{split}
\right.
\end{equation}
For the axisymmetric velocity field $\mathbf{u}$, we can also compute the vorticity $\boldsymbol{\omega}=\mathrm{curl}~\mathbf{u}$ as follows,
\begin{equation}\label{omega}
\boldsymbol{\omega}=\omega^{r}\mathbf{e}_{r}+\omega^{\theta}\mathbf{e}_{\theta}+\omega^{3}\mathbf{e}_{3},
\end{equation}
with $
\omega^{r}=-\partial_{3}u^{\theta},~\omega^{\theta}=\partial_{3}u^{r}-\partial_{r}u_{3},~\omega^{3}=\partial_{r}u^{\theta}+\frac{u^{\theta}}{r}$. Furthermore, $(\omega^{r},\omega^{\theta},\omega^{3})$ satisfy
\begin{equation}\label{NS-omega}
\left\{
\begin{split}
&\partial_{t}\omega^{r}+(u^{r}\partial_{r}+u^{3}\partial_{3})\omega^{r}-(\partial_{r}^{2}+\partial_{3}^{2}+\frac{1}{r}\partial_{r}-\frac{1}{r^{2}})\omega^{r}-(\omega^{r}\partial_{r}+\omega^{3}\partial_{3})u^{r}=0,\\
&\partial_{t}\omega^{\theta}+(u^{r}\partial_{r}+u^{3}\partial_{3})\omega^{\theta}-(\partial_{r}^{2}+\partial_{3}^{2}+\frac{1}{r}\partial_{r}-\frac{1}{r^{2}})\omega^{\theta}-\frac{2u^{\theta}\partial_{3}u^{\theta}}{r}-\frac{u^{r}\omega^{\theta}}{r}=0,\\
&\partial_{t}\omega^{3}+(u^{r}\partial_{r}+u^{3}\partial_{3})\omega^{3}-(\partial_{r}^{2}+\partial_{3}^{2}+\frac{1}{r}\partial_{r})\omega^{3}-(\omega^{r}\partial_{r}+\omega^{3}\partial_{3})u^{3}=0,\\
&(\omega^{r},\omega^{\theta},\omega^{3})|_{t=0}=(\omega_{0}^{r},\omega_{0}^{\theta},\omega_{0}^{3}).
\end{split}
\right.
\end{equation}
The bounded property of $ru^{\theta}$ preserves as the time grows, i.e. $ru^{\theta} \in L^{\infty}([0,+\infty);L^{\infty}(\R^{3}))$, if $ru_{0}^{\theta}\in L^{\infty}(\R^{3})$, see \cite{Pokorny1, JG.Liu1} etc. It is an essential ingredient for the axisymmetric Navier-Stokes equations.  And it makes us to consider the regularity criteria of $u^{\theta}$ in the critical case for the  axisymmetric Navier-Stokes equations.

We recall that global well-posedness result was firstly proved under no swirl assumption, i.e. $u^{\theta}=0$, independently by Ukhovskii and Yudovich \cite{Ukhovskii}, and Ladyzhenskaya \cite{Ladyzhenskaya}, also \cite{Leonardi} for a refined proof. When the angular velocity $u^{\theta}$ is not trivial, the global well-posedness problem is still open. Much attentions has been draw for decades and tremendous efforts and interesting progress have been made on the regularity problem of the axisymmetric Navier-Stokes equations\cite{Chae,Chen1,Chen2,Chen2017a,Chen2017b,Koch,
Pokorny3,Z.Lei,P.Zhang} etc. .
In \cite{Chen1,Chen2}, Chen, Strain, Tsai and Yau  proved that the suitable weak solutions are smooth if the velocity field $\mathbf{u}$ satisfies $r|\mathbf{u}|\leq C<\infty$. Applying the Liouville type theorem for the ancient solutions of Navier-Stokes equations, Z. Lei and Qi S. Zhang \cite{Z.Lei} obtained the similar result in the case $\mathbf{b}=u^{r}(t,r,x_{3})\mathbf{e_{r}}+u^{3}(t,r,x_{3})\mathbf{e_{3}}\in L^\infty((0,T);BMO^{-1})$. And we promote in \cite{Chen2017b} that the solution $\mathbf{u}$ is smooth in $(0,T]\times \R^{3}$, if $r^{d} u^{\theta} \in L^{p}((0,T);L^{q}(\R^{3}))$, where
$$
\frac{2}{p}+\frac{3}{q}\leq1-d,~0\leq d<1,~\frac{3}{1-d} < q\leq \infty, ~\frac{2}{1-d} \leq p\leq\infty.
$$
The above regularity criteria of $u^\theta$, which is  scaling invariant, greatly develop the corresponding regularity criteria in \cite{Pokorny2,Pokorny3,Pokorny1,P.Zhang}. Unfortunately, it fails in the critical case $d=1$, which is the ideal goal, since the conservation law of $ru^{\theta}$. Since then, there are some significant improvements and applications (\cite{Lei2017,Wei2016,Chen2017a,Le2017}).

In this paper, we introduce an anisotropic Lebesgue space $L^{p}_{T}L^{q_{v}}_{v}L^{q_{h},w}_{h}$, since the solutions behavior anisotropic on the variable $r$ and $x_{3}$. By applying the delicate \textit{a priori} estimations for the equations of $(\Phi,\Gamma)$, we can obtain the regularity criteria
$$
u^{\theta} \in L^{p}_{T}L^{q_{v}}_{v}L^{q_{h},w}_{h},~~\frac{2}{p}+\frac{1}{q_{v}}+\frac{2}{q_{h}}\leq 1, ~2<q_{h}\leq\infty,~\frac{1}{q_{v}}+\frac{2}{q_{h}}<1.
$$
It improves the regularity criteria in \cite{Chen2017b}. Moreover it provides  us a new perspective to the open problem, instead of the weighted Lebesgue space in \cite{Chen2017b}. For instace, we assume $ru^{\theta}$ is H\"{o}lder for the variable $r$, i.e. $|u^{\theta}|\leq Cr^{\alpha-1}, ~0< \alpha\leq 1.$ Therefore, the solution $\mathbf{u}$ is regular, since $u^{\theta}\in L_{T}^{\infty}L^{\infty}_{v}L^{\frac{2}{1-\alpha},w}_{h}$. The authors in \cite{Chen1, Chen2, Chen2017b,Z.Lei} drew a similar argument. And there are some detail discussions in Remark \ref{rem1} for the extreme points. \\

\textbf{Notations.} Throughout this paper,  $L^{q,r}(\R^{n})$ stands for Lorentz space, while $L^{q,w}=L^{q,\infty}$.\\
Moreover, we introduce the Banach space $L^{p}_{T}L^{q_{v}}_{v}L^{q_{h},w}_{h}$, equipped with norm
$$
\|f\|_{L^{p}_{T}L^{q_{v}}_{v}L^{q_{h},w}_{h}}=\|\|\|f(t,x_{1},x_{2},x_{3})\|_{L^{q_{h},w}(\R^{2},dx_{1}dx_{2})}\|_{L^{q_{v}}(\R,dx_{3})}\|_{L^{p}((0,T),dt)}.
$$
And we denote $\dot{H}^{s,p}$ and $\dot{B}_{p,q}^{s}$ for the homogeneous Soblev space and homogeneous Besov space, respectively. For simplicity, we denote $\dot{H}_{h}^{s}=\dot{H}^{s}(\R^{2},dx_{1}dx_{2}), L_{x}^{p}=L^{p}(\R^{3},dx)$. And the other ones are similar.

We note $\mathbf{b}=u^{r}(t,r,x_{3})\mathbf{e}_{r}+u^{3}(t,r,x_{3})\mathbf{e}_{3}$, and $(\Phi,\Gamma)=(\frac{\omega^{r}}{r},\frac{\omega^{\theta}}{r})$, while $\omega^{r},\omega^{\theta}$ is defined in (\ref{omega}).

Finally, we note C the arbitrary constant.

\section{Main Result}
\begin{thm}\label{thm2.1}
Let $\mathbf{u}\in C([0,T);H^2(\R^3))\cap L^2_{loc}([0,T);H^3(\R^3))$ be the unique axisymmetric solution of the Navier-Stokes equations with the axisymmetric initial data $\mathbf{u_{0}}\in H^{2}(\R^{3})$ and $\mathrm{div}~\mathbf{u_{0}}=0$. If $ru_{0}^{\theta}\in L^{\infty}$ and the time $T<\infty$,  the solution $\mathbf{u}$ can be continued beyond the time $T$, provided that the swirl $u^{\theta}$ satiesfies
$$
r^{d}u^{\theta}\in L^{p}_{T}L^{q_{v}}_{v}L^{q_{h},w}_{h},~~
\frac{2}{p}+\frac{1}{q_{v}}+\frac{2}{q_{h}}\leq 1-d, ~-1\leq d<1,\frac{2}{1-d}<q_{h}\leq\infty,~\frac{1}{q_{v}}+\frac{2}{q_{h}}<1-d.
$$
\end{thm}
Set $d=0$ in Theorem \ref{thm2.1}, and the following corollary is derived straight forward.
\begin{cor}\label{cor2.2}
Let $\mathbf{u}\in C([0,T);H^2(\R^3))\cap L^2_{loc}([0,T);H^3(\R^3))$ be the unique axisymmetric solution of the Navier-Stokes equations with the axisymmetric initial data $\mathbf{u_{0}}\in H^{2}(\R^{3})$ and $\mathrm{div}~\mathbf{u_{0}}=0$. If $ru_{0}^{\theta}\in L^{\infty}$ and the time $T<\infty$, the solution $\mathbf{u}$ can be continued beyond the time $T$, provided that the swirl $u^{\theta}$ satiesfies
\begin{equation}
u^{\theta} \in L^{p}_{T}L^{q_{v}}_{v}L^{q_{h},w}_{h},~~\frac{2}{p}+\frac{1}{q_{v}}+\frac{2}{q_{h}}\leq 1, ~2<q_{h}\leq\infty,~\frac{1}{q_{v}}+\frac{2}{q_{h}}<1.
\end{equation}
\end{cor}

\begin{rem}\label{rem1}
At the extreme points $\{p=\infty,~\frac{1}{q_{v}}+\frac{2}{q_{h}}=1,~2<q_{h}\leq\infty \}$ in Corollary \ref{cor2.2}, we can still derive regularity criteria with an additional smallness assumption
$$
\|u^{\theta}\|_{L^{\infty}_{T}L^{q_{v}}_{v}L^{q_{h},w}_{h}} \leq \epsilon,
$$
where $\epsilon$ is a sufficiently small constant. The precise proof can be dealed with in an analogous process in Section \ref{section4}.

As a matter of fact,  $u^{\theta} \in L^{\infty}_{T}L^{\infty}_{v}L^{2,w}_{h}$, since $|ru^{\theta}|\leq C \|ru_{0}^{\theta}\|_{L^{\infty}(\R^{3})}$. And the global regularity of the solutions of axisymmetric Navier-Stokes equations can be solved, if the extreme point $(p,q_{v},q_{h})=(\infty,\infty,2)$ in Corollary \ref{cor2.2} is settled.
However, the problems remain open. Recently, D. Wei \cite{Wei2016} estabilshed regularity criterion of the form $|u^{\theta}| \leq  \frac{C}{r|\ln{r}|^{\frac{3}{2}}}, r<\frac{1}{2} $. And the function $\frac{1}{r|\ln{r}|^{\frac{3}{2}}}|_{r<\frac{1}{2}}\in L^{\infty}_{T}L^{\infty}_{v}L^{2,\beta}_{h}, \beta>\frac{2}{3}$.
Therefore, it remains gaps between $L^{2,\frac{2}{3}}_{h}$ and $L^{2,w}_{h}$ in a certain sense.
\end{rem}
\begin{rem}
The Corollary \ref{cor2.2} still holds if we replace the regularity criteria by
\begin{equation}
u^{\theta}|_{r<\delta} \in L^{p}_{T}L^{q_{v}}_{v}L^{q_{h},w}_{h},
\end{equation}
where $\frac{2}{p}+\frac{1}{q_{v}}+\frac{2}{q_{h}}\leq 1, ~2<q_{h}\leq\infty,~\frac{1}{q_{v}}+\frac{2}{q_{h}}<1$ and $\delta >0$ is a arbitrary constant.

\end{rem}
Inspired by \cite{Chen1,Chen2,Koch,Z.Lei}, we have the following theorem in the critical space $L^{\infty}_{T}L^{\infty}_{v}L^{2,w}_{h}$.
\begin{thm}\label{thm2.4}
Let $\mathbf{u}$ be an axisymmetric  suitable weak solution of the Navier-Stokes equations (\ref{NS}) with the axisymmetric initial data $\mathbf{u_{0}}\in L^{2}(\R^{3})$, $\mathrm{div}~\mathbf{u_{0}}=0$, and $ru_{0}^{\theta}\in L^{\infty}(\R^{3})$. Suppose $\mathbf{b} \in L^{\infty}_{T}L^{\infty}_{v}L^{2,w}_{h}$, then $\mathbf{u}$  is smooth in $(0,T]\times \R^{3}$.
\end{thm}
\begin{rem}
The proof is analogously to \cite{Z.Lei}, and we omit the details here.
\end{rem}
\section{Preliminaries}
We will give some useful \textit{a priori} estimates in the axisymmetric Navier-Stokes equations, and refer to \cite{Chen2017b,JG.Liu1,Pokorny1,Wei2016} for details.
\begin{lem}\label{lem3.1}
Assume $\mathbf{u}$ is the smooth axisymmetric solution of (\ref{NS}) on $[0,T]$. If in addition, $ru_{0}^{\theta}\in L^{\infty}(\R^{3})$, then $|ru^{\theta}|\leq C \|ru_{0}^{\theta}\|_{L^{\infty}(\R^{3})}$.
\end{lem}
\begin{lem}[\cite{Wei2016}]\label{lem3.2}
\begin{equation}
\|\nabla\frac{u^{r}}{r}\|_{L^{2}(\R^{3})}\leq~\|\Gamma\|_{L^{2}(\R^{3})},~~\|\nabla^{2}\frac{u^{r}}{r}\|_{L^{2}(\R^{3})}\leq~\|\partial_{3}\Gamma\|_{L^{2}(\R^{3})}.
\end{equation}
\end{lem}

\begin{lem}\label{lem3.3}\cite{Chen2017b}

\noindent Let $\mathbf{u}\in C([0,T);H^2(\R^3))\cap L^2_{loc}([0,T);H^3(\R^3))$ be the unique axisymmetric solution of the Navier-Stokes equations with the axisymmetric initial data $\mathbf{u_{0}}\in H^{2}(\R^{3})$ and $\mathrm{div}~\mathbf{u_{0}}=0$.
If in addition, $T<\infty$ and $\|\Gamma\|_{L^{\infty}((0,T);L^{2}(\R^{3}))} < \infty$, then $\mathbf{u}$ can be continued beyond $T$.
\end{lem}

For convenience of readers, we will list some basic properties of Lorentz space.
\begin{lem}\label{Lorentz}
We denote $L^{p,q},0<p,q\leq \infty$ the Lorentz space.
\begin{enumerate}
\item[\rm{(i)}] If $0<p,r<\infty,0<q\leq \infty$,
\begin{equation}
\||g|^{r}\|_{L^{p,q}}=\|g\|_{L^{pr,qr}}^{r}.
\end{equation}

\item[\rm{(ii)}][pointwise product] Let $1<p<\infty,1\leq q \leq \infty,\frac{1}{p}+\frac{1}{p'}=1,\frac{1}{q}+\frac{1}{q'}=1$. Then pointwise multiplication is a bounded bilinear operator:
\begin{enumerate}
\item[\rm{a)}] from $L^{p,q}\times L^{\infty}$ to $L^{p,q}$;
\item[\rm{b)}] from $L^{p,q}\times L^{p',q'}$ to $L^{1}$;
\item[\rm{c)}] from $L^{p,q}\times L^{p_{1},q_{1}}$ to $L^{p_{2},q_{2}}$, for $1<p_{1},p_{2}<\infty,\frac{1}{p_{2}}=\frac{1}{p}+\frac{1}{p_{1}},\frac{1}{q_{2}}=\frac{1}{q}+\frac{1}{q_{1}}$;
\item[\rm{d)}] if $q<\infty$, the dual space of $L^{p,q}$ is $L^{p',q'}$.
\end{enumerate}
\item[\rm{(iii)}][convolution] Let $1<p<\infty,1\leq q \leq \infty,\frac{1}{p}+\frac{1}{p'}=1,\frac{1}{q}+\frac{1}{q'}=1$. Then convolution is a bounded bilinear operator:
\begin{enumerate}
\item[\rm{a)}] from $L^{p,q}\times L^{1}$ to $L^{p,q}$;
\item[\rm{b)}] from $L^{p,q}\times L^{p',q'}$ to $L^{\infty}$;
\item[\rm{c)}] from $L^{p,q}\times L^{p_{1},q_{1}}$ to $L^{p_{2},q_{2}}$, for $1<p_{1},p_{2}<\infty,1+\frac{1}{p_{2}}=\frac{1}{p}+\frac{1}{p_{1}},\frac{1}{q_{2}}=\frac{1}{q}+\frac{1}{q_{1}}$.
\end{enumerate}
\end{enumerate}
\end{lem}

We give a general Sobolev-Hardy-Littlewood inequality.
\begin{lem} \label{lem3.4}We assume $2\leq p<\infty,~0\leq s<\frac{n}{p},1\leq r\leq \infty$. For all $f\in \dot{B}_{2,r}^{s+n(\frac{1}{2}-\frac{1}{p})}(\R^{n})$, we have
\begin{equation}
\|\frac{f}{|x|^{s}}\|_{L^{p,r}(\R^{n})}\leq ~C~\|f\|_{\dot{B}_{2,r}^{s+n(\frac{1}{2}-\frac{1}{p})}(\R^{n})}.
\end{equation}
\end{lem}
\begin{proof}Set $\frac{1}{p}=\frac{s}{n}+\frac{1}{q}$, $p<q<\infty$. Apply Lemma \ref{Lorentz} and intepolation, successively, we have
\begin{align*}
\|\frac{f}{|x|^{s}}\|_{L^{p,r}(\R^{n})}&\leq ~C~ \|\frac{1}{|\cdot|^{s}}\|_{L^{\frac{n}{s},w}}\|f\|_{L^{q,r}(\R^{n})}\\
&\leq~ C~ \|f\|_{L^{q,r}(\R^{n})}\\
&\leq~ C~ \|f\|_{\dot{B}_{p,r}^{s}(\R^{n})}\\
&\leq ~C~ \|f\|_{\dot{B}_{2,r}^{s+n(\frac{1}{2}-\frac{1}{p})}(\R^{n})}.
\end{align*}
\end{proof}

\begin{lem}[Trace Operator]\label{lem3.5} For all $f\in \dot{H}^{1}(\R^{3})$, we have
\begin{equation}
\underset{x_{3}\in \R}{esssup}~~ \|f(\cdot,x_{3})\|_{\dot{H}^{\frac{1}{2}}(\R^{2})}\leq ~C~ \|f(\cdot)\|_{\dot{H}^{1}(\R^{3})}.
\end{equation}
\end{lem}
\begin{proof}
Since the translation invariance and dense embedding, it is sufficiently to show that
$$
\|f(\cdot,0)\|_{\dot{H}^{\frac{1}{2}}(\R^{2})}\leq ~C~ \|f(\cdot)\|_{\dot{H}^{1}(\R^{3})}, ~\text{for} ~~\forall f\in S(\R^{3}).
$$
Set $x'=(x_{1},x_{2}),~\xi'=(\xi_{1},\xi_{2})$ and $\gamma(f)(x_{1},x_{2})=f(x_{1},x_{2},0)=f(x',0)$,
\begin{align*}
\gamma(f)(x_{1},x_{2})&=~(2\pi)^{-3}\int\int\int e^{ix'\cdot\xi'}\hat{f}(\xi)~d\xi,\\
\widehat{\gamma(f)}(\xi_{1},\xi_{2})&=~(2\pi)^{-1}\int\hat{f}(\xi)~d\xi_{3}\\
&\leq~ C \int \hat{f}|\xi||\xi|^{-1}~d\xi_{3}\\
&\leq~ C \left(\int | \hat{f}|^{2}|\xi|^{2}~d\xi_{3}\right)^{\frac{1}{2}}\left(\int|\xi|^{-2}~d\xi_{3}\right)^{\frac{1}{2}}\\
&\leq~ C \left(\int | \hat{f}|^{2}|\xi|^{2}~d\xi_{3}\right)^{\frac{1}{2}} |\xi'|^{-\frac{1}{2}},
\end{align*}
Thus
$$
|\xi'||\widehat{\gamma{f}}|^{2}\leq C \int |\hat{f}|^{2} |\xi|^{2}~d\xi_{3}.
$$
Finally, integrating both side with $\int\int ~d\xi'$, we will derive the result.
\end{proof}
We present an essential \textit{a priori} estimate below for Theorem \ref{thm2.1}.
\begin{lem}\label{lem3.6}
Assume $\frac{2}{p}+\frac{1}{q_{v}}+\frac{2}{q_{h}}= 1-d,~~-1\leq d<1,\frac{2}{1-d}<q_{h}\leq\infty,~\frac{1}{q_{v}}+\frac{2}{q_{h}}< 1-d$. For a sufficiently small constant $\epsilon>0$, we have
\begin{equation}
\int_{\R^{3}} \frac{|u^{\theta}|}{r}|f|^{2}~dx\leq~ C_{\epsilon}\|r^{d}u^{\theta}\|_{L^{q_{v}}_{v}L^{q_{h},w}_{h}}^{p}\|f\|_{L^{2}(\R^{3})}^{2}+\epsilon\|\nabla f\|_{L^{2}(\R^{3})}^{2}.
\end{equation}
\end{lem}
\begin{proof}
Set
$$
a=\frac{2\tau}{p},~b=\frac{2\tau}{q_{v}},~c=2-a-b,~\frac{2}{\gamma}=1-\frac{\tau}{q_{h}}, ~~~\tau=\left\{\begin{split}&1,&0\leq d<1\\&\frac{1}{1-d},&-1\leq d <0 \end{split} \right..
$$
Thus
$$
0\leq a,b,c\leq 2,~a\neq0,~2\leq \gamma<4.
$$
It is appeared to see that $\gamma>2$ if $q_{h}<\infty$, and $\gamma=2$ if $q_{h}=\infty$.
By applying Lemma \ref{Lorentz}, Lemma \ref{lem3.1},  Lemma \ref{lem3.4}, Lemma \ref{lem3.5} and interpolation, we can deduce the following estimates, respectively.
\begin{align*}
\int_{\R^{3}} \frac{|u^{\theta}|}{r}|f|^{2}~dx&=\int_{\R^{3}}|(ru^{\theta})^{1-\tau}(r^{d}u^{\theta})^{\tau}| \frac{f^{2}}{r^{2+(d-1)\tau}}~dx\\
&\leq~C~ \int_{\R}\|r^{d}u^{\theta}\|^{\tau}_{L^{q_{h},w}_{h}}\||\frac{f}{r^{1+\frac{(d-1)\tau}{2}}}|^{2}\|_{L^{\frac{\gamma}{2},1}_{h}}~dx_{3}\\
&\leq~C~ \int_{\R}\|r^{d}u^{\theta}\|^{\tau}_{L^{q_{h},w}_{h}}\|\frac{f}{r^{1+\frac{(d-1)\tau}{2}}}\|^{2}_{L_{h}^{\gamma,2}}~dx_{3}\\
&\leq~C~\int_{\R}\|r^{d}u^{\theta}\|^{\tau}_{L^{q_{h},w}_{h}}\|f(\cdot,x_{3})\|_{\dot{B}_{2,2}^{2+\frac{(d-1)\tau}{2}-\frac{2}{\gamma}}(\R^{2})}^{2}~dx_{3}\\
&\leq~C~ \int_{\R}\|r^{d}u^{\theta}\|^{\tau}_{L^{q_{h},w}_{h}} \|f(\cdot,x_{3})\|^{a}_{L_{h}^{2}} \|f(\cdot,x_{3})\|^{b}_{\dot{H}_{h}^{\frac{1}{2}}}\|f(\cdot,x_{3})\|^{c}_{\dot{H}_{h}^{1}}~dx_{3}\\
&\leq~C~ \|r^{d}u^{\theta}\|^{\tau}_{L_{v}^{q_{v}}L_{h}^{q_{h},w}}\|f\|^{a}_{L^{2}_{x}}\|f\|^{b}_{L_{v}^{\infty}\dot{H}_{h}^{\frac{1}{2}}} \|f\|^{c}_{L^{2}_{v}\dot{H}_{h}^{1}}\\
&\leq~C~ \|r^{d}u^{\theta}\|^{\tau}_{L_{v}^{q_{v}}L_{h}^{q_{h},w}}\|f\|^{a}_{L^{2}_{x}}\|\nabla f\|^{2-a}_{L^{2}_{x}}\\
&\leq~C_{\epsilon}\|r^{d}u^{\theta}\|_{L^{q_{v}}_{v}L^{q_{h},w}_{h}}^{p}\|f\|_{L^{2}_{x}}^{2}+\epsilon\|\nabla f\|_{L^{2}_{x}}^{2}.
\end{align*}
\end{proof}

\section{Proof of Theorem \ref{thm2.1} \label{section4}}
$\bullet$ For simplicity, we only prove the Theorem \ref{thm2.1} in the critical case, i.e.
\begin{equation}
r^{d}u^{\theta}\in L^{p}_{T}L^{q_{v}}_{v}L^{q_{h},w}_{h},~~
\frac{2}{p}+\frac{1}{q_{v}}+\frac{2}{q_{h}}=1-d, ~-1\leq d<1,\frac{2}{1-d}<q_{h}\leq\infty,~\frac{1}{q_{v}}+\frac{2}{q_{h}}< 1-d.
\end{equation}
Otherwise, we can find $p_{*}<p$, such that
\begin{equation}
\begin{split}
&\|r^{d}u^{\theta}\|_{L^{p_{*}}_{T}L^{q_{v}}_{v}L^{q_{h},w}_{h}}\leq T^{\frac{1}{p_{*}}-\frac{1}{p}}\|r^{d}u^{\theta}\|_{L^{p}_{T}L^{q_{v}}_{v}L^{q_{h},w}_{h}}<\infty, \\
&\frac{2}{p_{*}}+\frac{1}{q_{v}}+\frac{2}{q_{h}}=1-d, ~-1\leq d<1,\frac{2}{1-d}<q_{h}\leq\infty,~\frac{1}{q_{v}}+\frac{2}{q_{h}}< 1-d.\label{4.2}
\end{split}
\end{equation}
Therefore, we can calculate analogously below with the condition (\ref{4.2}).

$\bullet$ As in \cite{Chen2017b}, we introduce the ingredient $(\Phi,\Gamma)=(\frac{\omega^{r}}{r},\frac{\omega^{\theta}}{r})$, which satisfy the following equations
 \begin{equation}\label{1.4}
\left\{
\begin{split}
\partial_{t}\Phi+(\mathbf{b}\cdot\nabla)\Phi-(\Delta+\frac{2}{r}\partial_{r})\Phi-(\omega^{r}\partial_{r}+\omega^{3}\partial_{3})\frac{u^{r}}{r}=0,\\
\partial_{t}\Gamma+(\mathbf{b}\cdot\nabla)\Gamma-(\Delta+\frac{2}{r}\partial_{r})\Gamma+2\frac{u^{\theta}}{r}\Phi=0.
\end{split}
\right.
\end{equation}
We show that
\begin{align}\label{4.4}
\frac{1}{2}\frac{d}{dt}\|\Phi\|_{L^{2}_{x}}^{2}+\|\nabla\Phi\|_{L^{2}_{x}}^{2}&=\int_{\R^{3}}u^{\theta}(\partial_{r}\frac{u^{r}}{r}\partial_{3}\Phi-\partial_{3} \frac{u^{r}}{r}\partial_{r}\Phi)~dx\nonumber\\
&\leq~\frac{1}{2}\int_{\R^{3}}|u^{\theta}|^{2}|\nabla \frac{u^{r}}{r}|^{2}~dx+\frac{1}{2}\|\nabla\Phi\|_{L^{2}_{x}}^{2}\nonumber\\
&\leq~C\int_{\R^{3}} \frac{|u^{\theta}|}{r} |\nabla \frac{u^{r}}{r}|^{2}~dx+\frac{1}{2}\|\nabla\Phi\|_{L^{2}_{x}}^{2}.
\end{align}
\begin{align}\label{4.5}
\frac{1}{2}\frac{d}{dt}\|\Gamma\|_{L^{2}_{x}}^{2}+\|\nabla\Gamma\|_{L^{2}_{x}}^{2}&=-2\int_{\R^{3}}\frac{u^{\theta}}{r}\Gamma\Phi~dx \nonumber\\
&\leq~\int_{\R^{3}}\frac{|u^{\theta}|}{r}|\Gamma|^{2}~dx+\int_{\R^{3}}\frac{|u^{\theta}|}{r}|\Phi|^{2}~dx.
\end{align}
Applying Lemma \ref{lem3.2} and Lemma \ref{lem3.6} in (\ref{4.4}), we have
\begin{align}
\frac{d}{dt}\|\Phi\|_{L^{2}_{x}}^{2}+\|\nabla\Phi\|_{L^{2}_{x}}^{2}&\leq~
C\int_{\R^{3}} \frac{|u^{\theta}|}{r}~|\nabla \frac{u^{r}}{r}|^{2}~dx \nonumber\\
&\leq~ C\|r^{d}u^{\theta}\|_{L^{q_{v}}_{v}L^{q_{h},w}_{h}}^{p}\|\nabla \frac{u^{r}}{r}\|_{L^{2}_{x}}^{2}+\frac{1}{4}\|\nabla^{2} \frac{u^{r}}{r}\|_{L^{2}_{x}}^{2}\nonumber\\
&\leq~C\|r^{d}u^{\theta}\|_{L^{q_{v}}_{v}L^{q_{h},w}_{h}}^{p}\|\Gamma\|_{L^{2}_{x}}^{2}+\frac{1}{4}\|\nabla\Gamma\|_{L^{2}_{x}}^{2}.\label{4.6}
\end{align}
Analogously, applying Lemma \ref{lem3.6} in (\ref{4.5}), it is easy to obtain that,
\begin{equation}\label{4.7}
\frac{1}{2}\frac{d}{dt}\|\Gamma\|_{L^{2}_{x}}^{2}+\|\nabla\Gamma\|_{L^{2}_{x}}^{2}
\leq C~\|r^{d}u^{\theta}\|_{L^{q_{v}}_{v}L^{q_{h},w}_{h}}^{p}\left(\|\Phi\|_{L^{2}_{x}}^{2}+\|\Gamma\|_{L^{2}_{x}}^{2}\right)+\frac{1}{4}\|\nabla\Phi\|_{L^{2}_{x}}^{2}+\frac{1}{4}\|\nabla\Gamma\|_{L^{2}_{x}}^{2}.
\end{equation}
Summing up (\ref{4.6}) and  (\ref{4.7}), we have
\begin{align*}
\frac{d}{dt}(\|\Phi\|_{L^{2}_{x}}^{2}+\|\Gamma\|_{L^{2}_{x}}^{2})+\|\nabla\Phi\|_{L^{2}_{x}}^{2}+\|\nabla\Gamma\|_{L^{2}_{x}}^{2}\leq~ C~ \|r^{d}u^{\theta}\|_{L^{q_{v}}_{v}L^{q_{h},w}_{h}}^{p}(\|\Phi\|_{L^{2}_{x}}^{2}+\|\Gamma\|_{L^{2}_{x}}^{2}).
\end{align*}
Using Gronwall's inequality, we have
\begin{equation}\label{3.9}
\sup_{t\in[0,T^*)} \|\Gamma\|_{L^{\infty}([0,t);L^{2}_{x})}^2\leq~ (\|\Phi_0\|_{L^{2}_{x}}^{2}+\|\Gamma_0\|_{L^{2}_{x}}^{2})\exp(C~\|u^{\theta}\|^{p}_{L^{p}_{T}L^{q_{v}}_{v}L^{q_{h},w}_{h}})<\infty.
\end{equation}
Applying Lemma \ref{lem3.3}, we obtain that $\mathbf{u}$ can be continued beyond $T$.

$\hfill\Box$

\section*{Acknowledgments}
 This work is
  partially supported by  NSF of
China under Grants 11671353,  11331005 and 11771389,
 Zhejiang Provincial Natural Science Foundation of China LR17A010001.





\end{document}